\documentclass[12pt]{amsart}
\usepackage{amssymb,amscd,amsmath,amsthm}
\usepackage{color}

\newtheorem{thm}{Theorem}

\newtheorem{lemma}[thm]{Lemma}
\newtheorem{remark}[thm]{Remark}
\newtheorem{example}[thm]{Example}

\begin{document}

\title[Terwilliger algebras of tournaments]{The Terwilliger algebras of doubly regular tournaments}

\author[A. Herman]{Allen Herman$^*$}
\address{Department of Mathematics and Statistics, University of Regina, Regina, Canada, S4S 0A2}
\email{Allen.Herman@uregina.ca}

\thanks{$^*$The author's research is supported by the Natural Sciences and Engineering Research Council of Canada. \\ The author has no conflicts of interest or competing interests to declare.   }
\date{\today}



\begin{abstract}
The Terwilliger algebras of asymmetric association schemes of rank $3$, whose nonidentity relations correspond to doubly regular tournaments, are shown to have thin irreducible modules, and to always be of dimension $4k+9$ for some positive integer $k$.   It is determined that asymmetric rank $3$ association schemes of order up to $23$ are determined up to combinatorial isomorphism by the list of their complex Terwilliger algebras at each vertex, but this no longer true at order $27$.  To distinguish order $27$ asymmetric rank $3$ association schemes, it is shown using computer calculations that the list of rational Terwilliger algebras at each vertex will suffice.     
\end{abstract}

\maketitle

\section{Introduction}

The goal of this article is to study the Terwilliger algebras of asymmetric rank $3$ association schemes.  Terwilliger introduced these as subconstituent algebras for arbitrary finite association schemes in \cite{T92}, and in \cite{T93b} he gave descriptions of their irreducible modules for the main families of $P$- and $Q$-polynomial association schemes with the thin property, and identified other families of these that do not have the thin property.  Since then, calculations of the irreducible modules for Terwilliger algebras have been given for strongly regular graphs \cite{TY94}, some group case association schemes (\cite{Balmaceda-Oura94},\cite{BM95}, \cite{HO2019}, and \cite{Maleki2024}), Doob schemes \cite{Tanabe97}, bipartite $P$- and $Q$-polynomial association schemes \cite{Caughman99}, and a few other cases.  Apart from these, there are general formulae for the Terwilliger algebras of direct products and wreath product association schemes (\cite{BST2010}, \cite{HKM2011}, \cite{MX2016}), but there are very few calculations for the Terwilliger algebras of other families of association schemes, in particular for asymmetric (i.e. non-symmetric) association schemes.  As asymmetric rank $3$ association schemes are a fundamental building block for asymmetric association schemes, they are a natural place to start.   

The main results of this paper show the Terwilliger algebras of asymmetric rank $3$ association schemes have thin irreducible modules, and that their dimensions are always equal to $9$ plus a multiple of $4$ and bounded by a function of the order of the scheme.  We note aslo that the same conclusions hold, by similar arguments, for the symmetric rank $3$ schemes generated by conference graphs. 

We then consider the extent to which an association scheme is determined up to combinatorial isomorphism by its Terwilliger algebras, a question that is particularly interesting for asymmetric rank $3$ schemes because any two of these with the same order are cospectral.  As many asymmetric rank $3$ association schemes $(X,S)$ have intransitive automorphism groups, several non-isomorphic Terwilliger algebras over $\mathbb{C}$ can occur as $x$ runs over the set of vertices $X$.  It turns out that the list of complex Terwilliger algebras $\{T_x : x \in X \}$ distinguishes asymmetric rank $3$ association schemes of order up to $23$, but this invariant does not distinguish those of order $27$.  This motivates us to consider the associated list of rational Terwilliger algebras for these schemes.  Using GAP \cite{GAP4} and the Terwilliger algebra features of its {\tt AssociationSchemes} package \cite{ASpkg}, we have verified that asymmetric rank $3$ association schemes of order $27$ are determined up to combinatorial isomorphism by their lists of rational Terwilliger algebras $\{\mathbb{Q}T_x : x \in X \}$ up to $\mathbb{Q}$-algebra isomorphism.  In order to analyze the algebraic structure of the rational Terwilliger algebras, we first establish a general result for arbitrary Terwilliger algebras, which shows any simple component of $\mathbb{Q}T_x$ that is associated with a thin irreducible $T_x$-module $V$ has rational Schur index $1$, and is hence isomorphic to a full matrix ring over the field of character values associated with $V$.  

\section{Asymmetric Association Schemes of rank $3$}

Let $\mathcal{A} = \{A_0=I_n, A_1, A_1^{\top} \}$ be the set of adjacency matrices of an asymmetric association scheme of rank $3$ and order $n$.  In particular, $A_1$ is the adjacency matrix of a total digraph on $n$ vertices satisfying $I_n + A_1 + A_1^{\top} = J_n$, where $A_1^{\top}$ denotes the transpose of $A_1$ and $J_n$ is the $n \times n$ matrix with every entry equal to $1$.   As is well known, this requires $A_1$ to be the usual adjacency matrix of a doubly regular tournament of order $n = 4u+3$, for some $u \ge 0$, and these matrices satisfy the identities 
\begin{equation}\label{eq1}
A_1A_1^{\top} = A_1^{\top}A_1= (2u+1)A_0 + u A_1 + u A_1^{\top},
\end{equation} 
and 
\begin{equation}\label{eq2}
A_1^2 = u A_1 + (u+1) A_1^{\top}.
\end{equation}
We record the following matrix equation for later use, which follows from the identities in (\ref{eq1}) and  (\ref{eq2}), and the fact that $I_n + A_1 + A_1^{\top} = J_n$:  if  $\alpha = \frac{-1+\sqrt{-n}}{2}$ and $\bar{\alpha}$ denotes its complex conjugate, then 
\begin{equation}\label{MatrixEq}
(A_1-\alpha I_n) (A_1 - \bar{\alpha} I_n) = (u+1)J_n. 
\end{equation}  

Regardless of whether or not they are realized by a collection of digraphs, asymmetric standard integral table algebras of rank $3$ exist for all $n = 4u+3$, $u \ge 0$.  This means that asymmetric rank $3$ association schemes are algebraically possible for all positive integers $n \equiv 3 \mod 4$.  The character table (with multiplicities) of such a table algebra takes the following form, where $\{b_0, b_1, b_1^*\}$ is the defining basis of the table algebra,  $\delta$ denotes the {\it valency} (or {\it degree}) character, and $\phi$ and $\bar{\phi}$ are a pair of complex conjugate irreducible characters:
$$\begin{array}{r|ccc|l} 
& b_0 & b_1 & b_1^* & \mbox{ multiplicity} \\ \hline
\delta & 1 & 2u+1 & 2u+1 & 1 \\
\phi & 1 & \alpha & \bar{\alpha} & 2u+1 \\
\bar{\phi} & 1 & \bar{\alpha} & \alpha & 2u+1 
\end{array}$$ 

Paley tournaments of this type are known to exist whenever $n$ is a prime congruent to $3$ mod $4$, these correspond to the fusion of the thin association scheme corresponding to the cyclic group whose non-trivial classes consist of squares and non-squares.  The number of non-isomorphic asymmetric rank $3$ schemes can be quite large, other than the Payley ones, they are almost always non-Schurian and often have trivial automorphism groups.  The number of asymmetric rank $3$ association schemes of each order $n \equiv 3 \mod 4$ for $n \le 31$ has been determined:  
$$\begin{array}{l|ccccccccc}
\mbox{ order }n & 3 & 7 & 11 & 15 & 19 & 23 & 27 & 31 & 35 \\ \hline
\mbox{number } & 1 & 1 & 1 & 1 & 2 & 19 & 374 & 98300 & ?
\end{array}$$  
The 98300 asymmetric rank $3$ association schemes were recently classified in \cite{HKMT-JCD2020}, they also report that the smallest $n \equiv 3 \mod 4$ for which such a scheme is not yet known to exist is $275$. 

\section{Terwilliger algbras over $\mathbb{C}$ and $\mathbb{Q}$} 

Let $\mathcal{A} = \{ A_0, A_1, \dots, A_d\}$ be the set of of adjacency matrices for an association scheme defined on the vertex set $X=\{1,\dots,n\}$.  For a fixed vertex $x \in X$, the set of the dual idempotents with respect to $x$ is the set  
$$\mathcal{E}^*(x) = \{E_0^*,E_1^*, \dots,E_d^*\}$$  
of $n \times n$ diagonal matrices defined as follows: for $1 \le u,v \le n$, 
\begin{equation}\label{eq4}
(E_i^*)_{uv} = \begin{cases}  1 & \mbox{ if } u=v \mbox{ and } (A_i)_{xu}=1 \\  0 & \mbox{ otherwise.} \end{cases}
\end{equation}
$\mathcal{E}^*(x)$ is a set of orthogonal and diagonal idempotents, whose ranks correspond to the valencies of the $A_i$, for $i=0,\dots,d$.  The {\it Terwilliger algebra} (a.k.a. {\it subconstitutent algebra}) of the association scheme with respect to a choice of vertex $x \in X$ is the subalgebra $T_x$ of $n \times n$ matrices over $\mathbb{C}$ generated by $\mathcal{A}$ and $\mathcal{E}^*(x)$.  This algebra is closed under the conjugate transpose, and is thus semisimple when considered as an algebra over $\mathbb{C}$.   We will also be interested in the rational Terwilliger algebra $\mathbb{Q}T_x$, which is the algebra generated by the sets of $n \times n$ integral matrices $\mathcal{A}$ and $\mathcal{E}^*(x)$ over the smaller field $\mathbb{Q}$.  The rational Terwilliger algebra $\mathbb{Q}T_x$ is a $\mathbb{Q}$-algebra with $T_x \simeq \mathbb{C} \otimes_{\mathbb{Q}} \mathbb{Q}T_x$, so its dimension over $\mathbb{Q}$ is the same as the dimension of $T_x$ over $\mathbb{C}$.  

The presence of the dual idempotents in the generating set of a Terwilliger algebra produces a natural refinement of its irreducible modules, a fact which has been exploited since its first introduction in \cite[\S 3]{T92}.  For every $E_i^* \in \mathcal{E}^*(x)$, $E_i^* T_x E_i^*$ is a subalgebra of $T_x$ whose multiplicative identity is $E_i^*$.  Given any irreducible $T_x$-module $V$, we have that 
$$V = \sum_{i=0}^d E_i^*V, $$ 
where each $E_i^*V$ is an $E_i^*T_xE_i^*$-module.  The {\it dual-diameter} of an irreducible $T_x$-module $V$ is the number of dual idempotents $E_i^*$ in $\mathcal{E}^*(x)$ for which $E_i^*V \ne 0$.  $V$ is said to be {\it thin} when all of these $E_i^*V$ are of dimension at most $1$.  

The {\it primary module} $V_0 = T_x e = T_x \hat{x}$ is the $T_x$-module generated by either the all $1$'s column vector $\mathbf{e}$ of length $n$, or the characteristic vector $\hat{x}$ of the vertex $x$.  The primary module gives an example of a thin irreducible $T_x$-module whose dimension $d+1$ is equal to its dual diameter \cite[Lemma 3.6]{T92}.  Note that the primary module is also realized over $\mathbb{Q}$, so it is an absolutely irreducible $\mathbb{Q}T_x$-module.  Since $E_0^* T_x E_0^*$ has dimension $1$, the other irreducible $T_x$-modules $V \ne V_0$ have $E_0^*V = 0$, so their dual-diameters will be less than $d+1$.  

If every irreducible $T_x$-module is thin, then we say that the Terwilliger algebra $T_x$ is thin.  Since the thinness of $T_x$ implies that every irreducible $E_i^*T_xE_i^*$-module has dimension $1$, this implies the components $E_i^*T_xE_i^*$ of the Pierce decomposition of $T_x$ with respect to $\mathcal{E}^*(x)$ are all commutative.  Conversely, commutativity of all the $E_i^*T_xE_i^*$ implies $T_x$ is thin \cite{T92}, so these notions are equivlaent.  This Pierce decomposition also makes sense for the rational Terwilliger algebra, so since $E_i^* (\mathbb{Q}T_x) E_i^* \subset E_i^* T_x E_i^*$, commutativity of these components over $\mathbb{C}$ implies they are also commutative over $\mathbb{Q}$.   

\section{The Subconstituent Decomposition of an asymmetric rank $3$ association scheme} 

Our approach to studying the irreducible modules of the Terwilliger algebras of asymmetric rank $3$ association schemes is similar to the approach used by Tomiyama and Yamizaki for the Terwilliger algebras of strongly regular graphs in \cite{TY94} that relied on subconstituent decompositions for these graphs introduced by Cameron, Goethals, and Seidel in \cite{CGS78}.   

Let $\mathcal{A} = \{ A_0, A_1, A_1^{\top} \}$ be the set of adjacency matrices for our asymmetric rank $3$ association scheme of order $n=4u+3$ on the set $X = \{1,\dots, n\}$.  By re-ordering the vertices, we may assume $x=1$ and the first row of $A_1$ has $1$'s in positions $2$ to $2u+2$.  Since $A_1$ is asymmetric, its' subconstituent decomposition takes this form: 
\begin{equation}\label{eq5}
A_1 = \sum_{i=0}^2 \sum_{j=0}^2 E_i^* A_1 E_j^* =  \begin{bmatrix} 0 & \mathbf{e}^{\top} & \mathbf{0}^{\top} \\ \mathbf{0} & B_1 & N \\ \mathbf{e} & J-N^{\top} & B_2  \end{bmatrix},
\end{equation} 
where $\mathbf{e}$ and $\mathbf{0}$ are the $(2u+1)$-column vectors of all $1$'s and all $0$'s respectively, $J$ is the $(2u+1) \times (2u+1)$ matrix of all $1$'s, and $B_1$, $B_2$, and $N$ are appropriate $(2u+1) \times (2u+1)$ $01$-matrices, with $B_1$ and $B_2$ asymmetric.  

Analyzing the implications of Equation (\ref{eq1}) and (\ref{eq2}) on the subconstituents of $A_1$ and $A_1^{\top}$, we obtain the following: 

$$ B_1 \mathbf{e} = u e, \mathbf{e}^{\top} B_1 = u \mathbf{e}^{\top}, $$ 
$$ N \mathbf{e} = (u+1) \mathbf{e}, \mathbf{e}^{\top}N=(u+1)\mathbf{e}^{\top}, $$ 
$$ B_2 \mathbf{e} = u \mathbf{e}, \mbox{ and } \mathbf{e}^{\top} B_2 = u \mathbf{e}^{\top}. $$ 
In particular, 
\begin{equation}\label{eq6} 
NJ = JN= (u+1)J, 
\end{equation} 
\begin{equation}\label{eq7} 
B_1J = JB_1 = uJ, \mbox{ and } 
\end{equation} 
\begin{equation}\label{eq8}
B_2 J =JB_2 = uJ.
\end{equation}
We can also see that 
$$ B_1^2+N(J-N^{\top}) = u B_1 + (u+1)B_1^{\top},  $$ 
$$ B_2^2+(J-N^{\top})N = u B_2 + (u+1)B_2^{\top}, $$
$$  B_1N + NB_2 = uN + (u+1)(J-N), \mbox{ and } $$ 
$$ J + (J-N^{\top})B_1 + B_2(J-N^{\top}) = u (J-N^{\top}) + (u+1) N^{\top}. $$
These imply 
\begin{equation}\label{eq9} B_1^2+B_1+(u+1)I = NN^{\top}, \end{equation} 
\begin{equation}\label{eq10} B_2^2+B_2+(u+1)I = N^{\top}N, \end{equation} 
\begin{equation}\label{eq11} N + B_1N + NB_2 = (u+1)J, \mbox{ and} \end{equation} 
\begin{equation}\label{eq12} B_2N^{\top} + N^{\top}B_1 + N^{\top} = (u+1)J. \end{equation}

\begin{lemma}\label{Bdiag}
$B_1^{\top} = J - I - B_1$ and $B_2^{\top} = J - I - B_2$ as $(2u+1) \times (2u+1)$ $01$-matrices.   In particular, $B_1$ and $B_2$ commute with their transposes, and are hence diagonalizable.  
\end{lemma} 

\begin{proof} 
The first two identities are a consequence of the fact that $A_1^{\top} = J - I - A_1$ as $n \times n$ $01$-matrices.  The last two conclusions follow from Equations (\ref{eq5}) and (\ref{eq8}), which show $B_1$ and $B_2$ commute with $J$. 
 \end{proof} 

\begin{thm}\label{thin}
Let $(X,S)$ be an asymmetric rank $3$ association scheme, and let $T_x$ be its complex Terwilliger algebra with respect to a fixed vertex $x \in X$.  Then $T$ is thin. 
 
\end{thm}

\begin{proof}
Suppose $(X,S)$ has order $n=4u+3$ for some $u \in \mathbb{N}$.  Assume without loss of generality that $x=1$ is the first vertex in $X$.  Let $\mathcal{A} = \{A_0, A_1, A_1^{\top} \}$ be the set of adjacency matrices of $(X,S)$ and let $\mathcal{E}^*(x) = \{E_0, E_1, E_2 \}$ be the set of dual idempotents corresponding to $x$.  Let $B_1$, $N$, and $B_2$ be the subconstituents of $A_1$ defined in Equation (\ref{eq5}).  

We claim that the corner subring $E_1^*T_xE_1^*$ is equal to the algebra generated by $I$, $B_1$, and $J$.  It is easy to see that $E_1^*T_xE_1^*$ will be generated as an algebra by certain words in the matrices $E_1^*$, $J_n$, $A_1$, and $E_2^*$, of increasing length: 
$$E_1^*, E_1^*J_nE_1^*, E_1^*A_1E_1^*, E_1^*A_1E_2^*A_1E_1^*, E_1^*A_1E_2^*A_1E_2^*A_1E_1^*, \mbox{ etc.} $$ 
Since $E_0^* = I_n - E_1^* - E_2^*$ and $A_1^{\top} = J_n - I_n - A_1$, it is only necessary to use words in the four matrices $E_1^*$, $A_1$, $J_n$, and $E_2^*$ that lie in $E_1^*T_xE_1^*$.  Under the natural projection to $(2u+1) \times (2u+1)$ matrices given by Equation (\ref{eq4}) that maps $E_1^*A_1E_1^*$ to $B_1$, the first three generators project to the $(2u+1) \times (2u+1)$ matrices $I$, $J$, and $B_1$.  We also have that $E_1^*A_1E_2^*A_1E_1^*$ projects to $N(J-N^{\top})$ and $E_1^*A_1E_2^*A_1E_2^*A_1E_1^*$ projects to $N B_2 (J-N^{\top})$.  Since 
$$N(J-N^{\top})=(u+1)J-B_1^2-B_1-(u+1)I)$$ 
and 
$$NB_2(J-N^{\top}) = -(u+1)J + B_1^3 +2B_1^2 +(u+2)B_1 + (u+1)I$$ 
by Equations (\ref{eq6}) through (\ref{eq12}), these will lie in the algebra generated by $I$, $J$, and $B_1$.  The same will be true of the projections of any generators of $E_1^*T_xE_1^*$ that is expressed using longer words in $\mathcal{E}^*$ and $\mathcal{A}$.  So $E_1^*T_xE_1^* \simeq \langle I, J, B_1 \rangle$.  Since $B_1$ commutes with $J$, this algebra is commutative.  A similar approach can be applied to show $E_2^*T_xE_2^* \simeq \langle I, J, B_2 \rangle$ and is commutative.   Since the three corner subalgebras $E_i^* T_x E_i^*$ for $i=0,1,2$ are all commutative, we can conclude from \cite{T92} that $T_x$ is thin.   
\end{proof}

\begin{thm}\label{main}
Let $(X,S)$ be an asymmetric rank $3$ association scheme of order $n = 4u+3$, for some $u \ge 0$.  Let $I_n = A_0, A_1, A_1^{\top}$ be its adjacency matrices.  Let $T_x$ be the complex Terwilliger algebra with respect to some $x \in X$.  Let $\{E_0^*, E_1^*, E_2^*\}$ be the set of dual idempotents with respect to $x$, and let $B_1$ be the subconstituent corresponding to $E_1^* A_1 E_1^*$.   

\begin{enumerate} 
\item Every non-primary irreducible $T_x$-module has dimension $1$ or $2$.

\item The number of distinct $1$-dimensional irreducible $T_x$-modules is either $0$ or $4$, and it is $4$ exactly when $\alpha = \frac{-1 + \sqrt{-n}}{2}$ is an eigenvalue of $B_1$. 

\item Let $d_{\alpha} \ge 0$ be the dimension of the $\alpha$-eigenspace of $B_1$.   Then the number of non-isomorphic $2$-dimensional irreducible $T_x$-modules is $m_2$, where $0 \le m_2 \le 2(u-d_{\alpha})$.  The dimension of $T_x$ is $9 + 4\epsilon + 4m_2$, where 
$$\epsilon = \begin{cases} 1 & \mbox{ if $T_x$ has a $1$-dimensional irreducible module} \\ 0 & \mbox{ otherwise} \end{cases}.$$ 
\end{enumerate}  
\end{thm} 

\begin{proof} (i) follows from Theorem \ref{thin}.  

It follows from Lemma \ref{Bdiag} that $B_1$ has a set of $2u+1$ linearly independent orthogonal eigenvectors.  From Equation (\ref{eq6}), one of these is $\mathbf{e}$.  Let $v$ be one of the $2u$ eigenvectors in this set that is orthogonal to $\mathbf{e}$, and suppose $\theta \in \mathbb{C}$ with $B_1 v = \theta v$.   Since $Jv$ will be $\mathbf{0}$, we will have 
$$A_1 \begin{bmatrix} 0 \\ v \\ \mathbf{0} \end{bmatrix} = 
\begin{bmatrix} 0 & \mathbf{e}^{\top} & \mathbf{0}^{\top} \\ \mathbf{0} & B_1 & N \\ \mathbf{e} & J-N^{\top} & B_2  \end{bmatrix}\begin{bmatrix} 0 \\ v \\ \mathbf{0} \end{bmatrix} = \begin{bmatrix} 0 \\ B_1 v \\ (J-N^{\top})v \end{bmatrix} =
\begin{bmatrix} 0 \\ \theta v \\ - N^{\top} v \end{bmatrix}. $$
From Equations (\ref{eq9}) and (\ref{eq12}), we obtain
$$\begin{array}{rcl} 
A_1 \begin{bmatrix}  0 \\ \mathbf{0} \\ -N^{\top}v \end{bmatrix} &=& \begin{bmatrix} 0 \\ -NN^{\top} v \\ - B_2N^{\top} v \end{bmatrix} \\
& & \\
&=& \begin{bmatrix} 0 \\ -(B_1^2+B_1+(u+1)I)v \\ ((u+1)J -N^{\top} - N^{\top}B_1) v \end{bmatrix} \\
& & \\
&=& \begin{bmatrix} 0 \\ -(\theta^2+\theta+(u+1))v \\ - (1 + \theta) N^{\top}v \end{bmatrix}.
\end{array}$$
Since these two vectors are fixed by elements of $\mathcal{E}^*$, we can see that $\begin{bmatrix} 0 \\ v \\ \mathbf{0} \end{bmatrix}$ and $\begin{bmatrix} 0 \\ \mathbf{0} \\ -N^{\top} v \end{bmatrix}$ generate an irreducible $T_x$-module.  Call this module $V_{1,\theta}$.  This module will have dimension $2$ exactly when $N^{\top}v \ne \mathbf{0}$.  Whenever $v$ is an eigenvector for $B_1$ orthogonal to $\mathbf{e}$ with eigenvalue $\theta$, Equations (\ref{eq9}) and (\ref{eq12}) tell us that 
$$B_2 N^{\top} v = (-1-\theta) N^{\top}v \mbox{ and } NN^{\top}v = (\theta^2 + \theta + (u+1))v.$$  
So $N^{\top}v = \mathbf{0}$ implies $\theta^2 + \theta + (u+1) = 0$, and hence $\theta \in \{ \alpha, \bar{\alpha} \}$.  When $\theta \not\in \{ \alpha, \bar{\alpha} \}$, we will have $NN^{\top}v \ne \mathbf{0}$, so $N^{\top}v \ne \mathbf{0}$, and so the module $V_{1,\theta}$ is $2$-dimensional. 

In a similar fashion, if we start with an eigenvector $w$ for $B_2$ with eigenvalue $\phi$ that is orthogonal to $\mathbf{e}$, we will have 
$$ A_1 \begin{bmatrix} 0 \\ \mathbf{0} \\ w \end{bmatrix} = \begin{bmatrix} 0 \\ Nw \\ B_2w \end{bmatrix} =
\begin{bmatrix} 0 \\ Nw \\ \phi w \end{bmatrix}, $$ 
so from Equations (\ref{eq10}) and (\ref{eq11}) we will get 
$$\begin{array}{rcl} 
A_1 \begin{bmatrix} 0 \\ Nw \\ \mathbf{0} \end{bmatrix} &=& \begin{bmatrix} 0 \\ B_1N w \\ (J-N^{\top})Nw \end{bmatrix} \\
& & \\
&=& \begin{bmatrix} 0 \\ ((u+1)J-N-N B_2) w \\ -N^{\top}Nw \end{bmatrix} \\
& & \\
&=& \begin{bmatrix} 0 \\ (-1-\phi) Nw \\ -(\phi^2+\phi+(u+1))w \end{bmatrix}.
\end{array} $$ 
So again we get an irreducible $T_x$-module $V_{2,\phi}$ generated by $\begin{bmatrix} 0 \\ \mathbf{0} \\ w \end{bmatrix}$ and $\begin{bmatrix} 0 \\ Nw \\ \mathbf{0}  \end{bmatrix}$, which will be $1$-dimensional when $\phi \in \{\alpha, \bar{\alpha} \}$ and otherwise $2$-dimensional.  In the $2$-dimensional case, $B_1Nw = (-1-\phi) w$, so we will have $V_{2,\phi} = V_{1,-1-\phi}$.  

Since $B_1$ is a matrix with rational entries, $\alpha$ will be an eigenvalue of $B_1$ if and only if $\bar{\alpha}$ is an eigenvalue of $B_1$, and the dimension $d_{\alpha}$ of the $\alpha$-eigenspace of $B_1$ will be equal to the dimension of the $\bar{\alpha}$-eigenspace of $B_1$.  It follows that, among the $2u$ eigenvectors of $B_1$ orthogonal to $\mathbf{e}$ we started with, $2u-2d_{\alpha}$ of them have eigenvalues $\theta$ that are not equal to $\alpha$ or $\bar{\alpha}$, and therefore their corresponding irreducible $T_x$-modules $V_{1,\theta} = V_{2,-1-\theta}$ will be $2$-dimensional.  (Note that $\bar{\alpha}=-1-\alpha$, so $\theta \not\in \{\alpha, \bar{\alpha} \}$ if and only if $-1-\theta \not\in \{ \alpha, \bar{\alpha} \}$.)  

Let $S$ be the standard $T_x$-module of dimension $n$.  Since 
$$V_{1,\theta} = E_1^*V_{1,\theta} \oplus E_2^*V_{1,\theta} = E_1^*V_{2,-1-\theta} \oplus E_2^*V_{2,-1-\theta},$$ 
with each of these components having dimension $1$, the sum of the dimensions in $E_1^*S$ and $E_2^*S$ that correspond to components of $2$-dimensional irreducible submodules of $S$ will be the same, and both of these will be $2u-2d_{\alpha}$.  This leaves a $2d_{\alpha}$-dimensional submodule of $E_2^*S$ that is filled with $1$-dimensional irreducible $T_x$-modules, and which projects to the direct sum of the $\alpha$- and $\bar{\alpha}$-eigenspaces of $B_2$.  The upshot of this is that whenever $\alpha$ is an eigenvalue of $B_1$, then $T_x$ has four $1$-dimensional irreducible modules $V_{1,\alpha}$, $V_{1,\bar{\alpha}}$, $V_{2,\alpha}$, and $V_{2,\bar{\alpha}}$, and if $\alpha$ is not an eigenvalue of $B_1$, then every non-primary irreducible $T_x$-module will be $2$-dimensional.   

Now, if $\beta: V_{1,\alpha} \rightarrow W$ is any non-zero $T_x$-module homomorphism, we would have $E_1^* \beta(v) = \beta (E_1^*v) \ne 0$ and $A_1 \beta(v) = \beta(A_1v) = \alpha \beta(v)$.  So $\beta(V_{1,\alpha})$ is a $1$-dimensional irreducible that is isomorphic to a component of $E_1^*S$, and $\beta(v)$ lies in the $\alpha$-eigenspace of $E_1^*A_1E_1^*$.  Therefore, $\beta(V_{1,\alpha}) \simeq V_{1,\alpha}$.  It follows then that the four $1$-dimensional irreducible $T_x$-modules $V_{1,\alpha}$, $V_{1,\bar{\alpha}}$, $V_{2,\alpha}$, and $V_{2,\bar{\alpha}}$ are pairwise non-isomorphic.  Since every irreducible $T_x$-module is isomorphic to a submodule of the standard $T_x$-module $S$, these are all of the $1$-dimensional irreducible $T_x$-modules up to isomorphism.  This proves (ii). 

By the preceding argument, the sum of the dimensions of the $2$-dimensional irreducible $T_x$-submodules of $S$ will be $4(u-d_{\alpha})$.  Since these submodules include representatives for all the $2$-dimensional irreducible $T_x$-modules, the number $m_2$ of isomorphism classes of these will be at most $2u-2d_{\alpha}$.  This proves (iii). 
\end{proof}

Continuing with the notation of the previous theorem, the next lemma shows we can determine the number of $2$-dimensional irreducible $T_x$-modules from the spectrum of $B_1$. 

\begin{lemma}\label{distinct}
Suppose $\theta_1 \ne \theta_2$ are eigenvalues of $B_1$ whose eigenvectors are orthogonal to $\mathbf{e}$.  If $\theta_1, \theta_2 \not\in \{ \alpha, \bar{\alpha} \}$ then the $2$-dimensional irreducible $T_x$-modules $V_{1,\theta_1}$ and $V_{1,\theta_2}$ are not isomorphic.   
\end{lemma}

\begin{proof} 
Suppose $\beta : V_{1,\theta_1} \rightarrow V_{1,\phi}$ is a $T_x$-module isomorphism between a pair of $2$-dimensional irreducible submodules of the standard module of $T_x$.  Then 
$$ (E_1^* A_1 E_1^*) \beta \bigg( \begin{bmatrix} 0 \\ v \\ 0 \end{bmatrix} \bigg) = 
\beta \bigg( E_1^*A_1E_1^* \begin{bmatrix} 0 \\ v \\ 0 \end{bmatrix} \bigg) = 
\beta \bigg( \theta \begin{bmatrix} 0 \\ v \\ 0 \end{bmatrix} \bigg) = 
\theta \beta \bigg( \begin{bmatrix} 0 \\ v \\ 0 \end{bmatrix} \bigg). $$ 
It follows that $\beta E_1^* \bigg( \begin{bmatrix} 0 \\ v \\ 0 \end{bmatrix} \bigg)$ is an eigenvector of $E_1^*A_1E_1^*$ with eigenvalue $\theta$.  Therefore, $\beta(V_{1,\theta})$ is a $2$-dimensional irreducible $T_x$-module corresponding to an element of the $\theta$-eigenspace of $B_1$, and so $\beta(V_{1,\theta}) \simeq V_{1,\theta}$. 
\end{proof} 

\begin{example} {\rm The above results tell us the structure of the complex Terwilliger algebra $T_x$ of an asymmetric rank $3$ association scheme at a given vertex $x$ determined by the spectrum of $B_1$, and hence by $E_1^*A_1E_1^*$.  $B_1$ has $2u$ eigenvalues other than $u$.  If all $2u$ of these are distinct, none are repeated, and none are equal to $\alpha$, then $\dim(T_x)$ will be the maximum possible, $9+4(2u)$, and otherwise it will be smaller.  Computationally the structure of the complex Terwilliger algebra can be deduced from the degrees of the distinct nonzero irreducible factors of the minimal polynomial of $E_1^*A_1E_1^*$, and the irreducible decomposition of the standard module can be deduced from the irreducible factorization of the characteristic polynomial of $E_1^*A_1E_1^*$.  For small $n$ the dimensions of these Terwilliger algebras $T_x$, as $x$ runs over all vertices, is as follows: 

{\footnotesize
$$\begin{array}{rl|rl}
\mbox{ AS } & \mbox{ $\dim(T_x)$, with frequency } & \mbox{ AS } & \mbox{ $\dim(T_x)$, with frequency }  \\ \hline 
\mbox{ {\tt as7no2} } & (17^7)  & & \\
\mbox{ {\tt as11no2} } & (25^{11}) & & \\ 
\mbox{ {\tt as15no5} } & (33^7,17^7) & & \\ 
\mbox{ {\tt as19no2} } & (41^{12},25^7) & \mbox{ {\tt as19no3} } & (41^{19}) \\ 
\mbox{ {\tt as23no2} } & (49^9,45^1,41^3,33^9,17^1) & \mbox{ {\tt as23no3} } & (49^{10},45^1,41^2,33^9,17^1) \\ 
\mbox{ {\tt as23no4} } & (49^8,41^5,33^9,17^1) & \mbox{ {\tt as23no5} } & (49^8,45^2,41^3,33^{10}) \\
\mbox{ {\tt as23no6} } & (49^{12},41^2,33^7,17^2) & \mbox{ {\tt as23no7} } & (49^7,45^1,41^5,33^9,17^1) \\ 
\mbox{ {\tt as23no8} } & (49^9,45^6,41^1,33^7) & \mbox{ {\tt as23no9} } & (49^7,45^5,41^1,33^9,17^1) \\ 
\mbox{ {\tt as23no10} } & (49^8,45^7,41^2,33^5,17^1) & \mbox{ {\tt as23no11} } & (49^{11},45^5,33^6,17^1) \\
\mbox{ {\tt as23no12} } & (49^{11},45^5,41^1,33^6) & \mbox{ {\tt as23no13} } & (49^{5},45^2,41^5,33^{10},17^1) \\
\mbox{ {\tt as23no14} } & (49^{12},45^4,41^1,33^{6}) & \mbox{ {\tt as23no15} } & (49^{10},45^5,41^1,33^{7}) \\
\mbox{ {\tt as23no16} } & (49^{10},45^6,33^{6},17^1) & \mbox{ {\tt as23no17} } & (49^{7},45^6,41^2,33^{8}) \\
\mbox{ {\tt as23no18} } & (49^{15},45^2,33^{6}) & \mbox{ {\tt as23no19} } & (49^{11},33^{11},17^1) \\
\mbox{ {\tt as23no20} } & (49^{23}) & \\ 
\end{array} $$
}
} 
\end{example}

In these small examples, the only Terwilliger algebras that have $1$-dimensional irreducible modules are the ones of dimension $45$ that occur for just a few vertices of some of the order $23$ schemes.  When $n=4u+3$, the maximum possible dimension $9+4(2u)$ occurs for at least one vertex in all of these small examples, and the minimum we see is $17=9+4(2)$.  Since $B_1$ is asymmetric nonnegative irreducible matrix, it should have at least three complex eigenvalues, one of them being a complex conjugate pair, so this is the smallest possible dimension that can occur.    

From this data we can see that asymmetric rank $3$ association schemes of order up to order $23$ are distinguished up to combinatorial isomorphism by their complex Terwilliger algebras.  For asymmetric rank $3$ schemes of order $27$ this is no longer the case.  In fact, not a single one of the $374$ combinatorial isomorphism classes of these schemes is determined by its list of complex Terwilliger algebras!  In the next section we will explain how to use the rational Terwilliger algebras to distinguish these schemes.

\begin{remark} {\rm A similar approach to Theorem \ref{main} gives an analogous conclusion for the dimensions of the Terwilliger algebras of symmetric rank $3$ association schemes generated by a conference graph.  This slightly strengthens the conclusions obtained for strongly regular graphs by Tomiyama and Yanazaki \cite{TY94} in the case of conference graphs. When the conference graph has order $n = 4u+1$, the character table of these association schemes has the form 
$$\begin{array}{r|ccc|l} 
& A_0 & A_1 & A_2 & \\ \hline
\delta & 1 & 2u & 2u & 1 \\
\psi & 1 & \beta & \beta^{\sigma} & 2u \\
\bar{\psi} & 1 & \beta^{\sigma} & \beta & 2u 
\end{array}$$ 
where $\beta = \frac{-1+\sqrt{4u+1}}{2}$ and $\beta^{\sigma}$ is its Galois conjugate.  The subconstituent decomposition of $A_1$ takes the form  
$$ A_1  = \begin{bmatrix} 0 & \mathbf{e}^{\top} & \mathbf{0}^{\top} \\ \mathbf{e} & B_1 & N \\ \mathbf{0} & N^{\top} & B_2  \end{bmatrix},$$
and working with the identity $A_1^2=(2u)I + (u-1)A_1 + uA_2$ we can deduce the identities 
$$ NN^{\top} = uJ - (B_1^2+B_1-uI) \mbox{ and } B_2N^{\top} = u J - N^{\top} - N^{\top} B_1. $$ 
From this, we can see that whenever $v$ is an eigenvector for $B_1$ that is orthogonal to $\mathbf{e}$ and has eigenvalue $\theta$, then $NN^{\top}v = -(\theta^2+\theta-u) v$ and $B_2N^{\top}v = (-1-\theta)v$.  As in the proof of Theorem \ref{main}, when $\theta \not\in \{ \beta, \beta^{\sigma} \}$, then $T_x$ will have a $2$-dimensional irreducible module generated by $\begin{bmatrix} 0 \\ v \\ 0 \end{bmatrix}$ and $\begin{bmatrix} 0 \\ N^{\top} v \end{bmatrix}$, and when $\theta \in \{\beta, \beta^{\sigma}\}$, then $\begin{bmatrix} 0 \\ v \\ 0 \end{bmatrix}$ generates a $1$-dimensional irreducible $T_x$-module.  As above, the number of non-isomorphic irreducible $T_x$-modules of each dimension is determined by the degrees of the distinct factors of the minimal polynomial of $B_1$.  

Since $B_1$ is an integral matrix, when it has the eigenvalue $\beta$, it also has the eigenvalue $\beta^{\sigma}=-1-\beta$.  Since the $2$-dimensional irreducible $T_x$-modules are also thin in this case, we can deduce that the number of $1$-dimensional irreducible $T_x$-submodules of the standard module that are generated from eigenvectors of $B_1$ is the same as that generated using eigenvectors of $B_2$.  So from this we can conclude the number of non-isomorphic $1$-dimensional irreducible $T_x$-modules is either $0$ or $4$.  So the dimension of $T_x$ is again always $9$ plus a multiple of $4$.  
}\end{remark} 

\section{ The Rational Terwilliger Algebra } 

In this section, we consider the rational Terwilliger algebra $\mathbb{Q}T_x$ of an association scheme $(X,S)$ of order $n$ with respect to a fixed vertex $x \in X$.  So $\mathbb{Q}T_x$ is the $\mathbb{Q}$-subalgebra of $M_n(\mathbb{Q})$ generated by the union $\mathcal{A} \cup \mathcal{E}^*(x)$, where $\mathcal{A} = \{A_0=I_n, A_1, \dots, A_d \}$ is the set of adjacency matrices of $(X,S)$ and $\mathcal{E}^*(x) = \{E_0^*,E_1^*, \dots, E_d^*\}$ is its set of dual idempotents with respect to $x$.  Since these are $01$-matrices, it is clear that these sets will generate an algebra with a basis whose structure constants are rational.  (Finding a basis of $T_x$ with integral structure constants turns out to be much more difficult, see \cite{Hanaki-ModularTerAlg}.) 

It is easy to see that $\mathbb{Q}T_x$ is a semisimple $\mathbb{Q}$-algebra.  If $\mathbb{Q}T_x$ were not semisimple, it would have a nonzero nilpotent ideal $I$, and $\mathbb{C} \otimes_\mathbb{Q} I$ would be a nonzero nilpotent ideal of $T_x$, which does not exist.  Because it is a finite-dimensional semisimple $\mathbb{Q}$-algebra, the algebraic structure of $\mathbb{Q}T_x$ is dependent upon the Galois conjugacy classes of irreducible $T_x$-modules, their fields of character values, and their associated Schur indices.  (For an overview of representation theory over non-algebraically closed fields of characteristic zero relevant to the discussion that follows, see \cite[\S 74]{CRvol2}).   

Since $T_x$ is semisimple $\mathbb{C}$-algebra with a rational basis, $\mathbb{Q}T_x$ has a splitting field $K \subset \mathbb{C}$ which is a finite Galois extension of $\mathbb{Q}$.  This splitting field $K$ satisfies  
$$ K \otimes_{\mathbb{Q}} \mathbb{Q}T_x = KT_x \simeq \displaystyle{\bigoplus_{\chi \in Irr(T_x)}} KT_x e_{\chi} \simeq \displaystyle{\bigoplus_{\chi \in Irr(T_x)}} M_{n_{\chi}}(K), $$
where $e_{\chi}$ is the centrally primitive idempotent of $T_x$ corresponding to each irreducible character $\chi \in Irr(T_x)$, and $n_{\chi}$ denotes the degree of $\chi$.   When we let $Gal(K/\mathbb{Q})$ act on $KT_x$, we obtain each centrally primitive idempotent $e_{\tilde{\chi}}$ of $\mathbb{Q}T_x$ as the sum over one Galois conjugate class of idempotents $e_{\chi}^{\sigma}$ as $\sigma$ runs over $Gal(K/\mathbb{Q})$.  Thus 
$$ \mathbb{Q}T_x \simeq \displaystyle{\bigoplus_{\tilde{\chi}}} \mathbb{Q}T_x e_{\tilde{\chi}}$$ 
as $\tilde{\chi}$ runs over the distinct Galois conjugacy classes of irreducible characters of $T_x$.  Each simple component $\mathbb{Q}T_x e_{\tilde{\chi}}$ is a central simple algebra whose dimension over its center is $n_{\chi}$.  This means $\mathbb{Q}T_x e_{\tilde{\chi}} \simeq M_r(D)$ where $D$ is a finite dimensional division algebra over $\mathbb{Q}$ satisfying $n_{\chi}^2 = r^2 [D:Z(D)]$. The dimension of $D$ over its center is the square of the {\it Schur index} of $D$.   Since $K$ has characteristic zero, the center of $\mathbb{Q}T_xe_{\tilde{\chi}}$ is $\mathbb{Q}$-isomorphic to the field of character values $\mathbb{Q}(\chi)$.  We remark that $\mathbb{Q}(\chi)/\mathbb{Q}$ need not be a normal extension, so the subfields of $K$ that are isomorphic to $\mathbb{Q}(\chi)$ are the images of $\mathbb{Q}(\chi)$ under elements of $Gal(K/\mathbb{Q})$, and so there will be $[\mathbb{Q}(\chi):\mathbb{Q}]$ of these embeddings $\mathbb{Q}(\chi) \hookrightarrow K$.  (Unlike what happens for rational adjacency algebras of association schemes, where it is an open question whether the splitting fields are always cyclotomic, we have found it to be quite common for the splitting fields of rational Terwilliger algebras to be noncyclotomic extensions of $\mathbb{Q}$.) 

If $W$ is an irreducible $\mathbb{Q}T_x$-module, then there will be a unique $e_{\tilde{\chi}}$ for which $e_{\tilde{\chi}}W \ne 0$, and the corresponding representation $\mathbb{Q}T_x \rightarrow End_{\mathbb{Q}}(W)$ has image isomorphic to $\mathbb{Q}T_xe_{\tilde{\chi}}$.  Conversely, given $\chi \in Irr(T_x)$, there is a unique centrally primitive idempotent $e_{\tilde{\chi}}$ of $\mathbb{Q}T_x$.  The simple algebra $\mathbb{Q}T_xe_{\tilde{\chi}}$ has just one irreducible module $W$ up to isomorphism, which lifts to an irreducible $\mathbb{Q}T_x$-module.  Tensoring this module with $K$ gives $KT_x$-module $K \otimes_{\mathbb{Q}} W$ that is isomorphic to a nonzero multiple $m'\tilde{V'}$ of $\tilde{V'}$, where $\tilde{V'}$ is the sum of the Galois conjugates of $V'$, an absolutely irreducible $KT_x$-module.  In this case $V = \mathbb{C} \otimes_K V'$ will be an irreducible $T_x$-module whose character is a Galois conjugate of $\chi$.   The multiplicity $m'$ is called the {\it rational Schur index} of the irreducible module $V$ (or of the irreducible character $\chi$); when $\mathbb{Q}T_x e_{\tilde{\chi}} \simeq M_r(D)$ above, then $m' = \sqrt{[D:\mathbb{Q}(\chi)]}$ agrees with the Schur index of the division algebra $D$.       

The goal of this section is to prove the rational Schur index associated with a thin $T_x$-module is $1$.  

\begin{thm}\label{ThinSchurIndex}
Suppose $V$ is a {\it thin} irreducible $T_x$-module whose character is $\chi \in Irr(T_x)$.  Then $\chi$ has rational Schur index equal to $1$.  In particular, the simple component of $\mathbb{Q}T_x$ determined by $V$ and $\chi$ will be isomorphic to $M_{d'}(\mathbb{Q}(\chi))$, where $d'$ is the dual diameter of $V$. 
\end{thm} 

\begin{proof} 
Let $\mathbb{Q}T_x e_{\tilde{\chi}}$ be the simple component of $\mathbb{Q}T_x$ determined by the given irreducible character $\chi$ of $V$, and let $W$ be an irreducible $\mathbb{Q}T_x$-module for which $V$ is a constituent of $\mathbb{C} \otimes_{\mathbb{Q}} W$.    Let $E_i^* \in \mathcal{E}^*(x)$ be one of the dual idempotents for which $E_i^*V \ne 0$.  Since $E_i^*W = 0$ implies $E_i^* V = 0$, we must have that $E_i^*W \ne 0$.  Since $e_{\tilde{\chi}}$ acts as the identity on $W$, we have $E_i^* e_{\tilde{\chi}} W \ne 0$.  Therefore, $E_i^* (\mathbb{Q}T_x e_{\tilde{\chi}}) E_i^*$ is non-zero.  But $\mathbb{Q}T_x e_{\tilde{\chi}}$ is a finite-dimensional simple algebra, and $E_i^* e_{\tilde{\chi}}$ is a non-zero idempotent of it, so we have that $E_i^* (\mathbb{Q}T_x e_{\tilde{\chi}}) E_i^*$ is Morita equivalent to $\mathbb{Q}T_x e_{\tilde{\chi}}$.  Therefore, if $\mathbb{Q}T_x e_{\tilde{\chi}} \simeq M_r(D)$ for some division algebra $D$ and positive integer $r$ satisfying $dim(V) = r [D:\mathbb{Q}(\chi)]$, then $E_i^* (\mathbb{Q}T_x e_{\tilde{\chi}}) E_i^* \simeq M_{r'}(D)$ with $1 \le r' \le r$.  

When $V$ is thin, we have that $E_i^* T_x E_i^*$ is commutative, so its subring $E_i^* (\mathbb{Q}T_x e_{\tilde{\chi}}) E_i^*$ is commutative.   Since it is also simple, it is a field.  Being Morita equivalent to a central simple algebra over $\mathbb{Q}(\chi)$, the field $E_i^* (\mathbb{Q}T_x e_{\tilde{\chi}}) E_i^*$ must be isomorphic to $\mathbb{Q}(\chi)$.  It follows that the rational Schur index of $V$ is $1$, and therefore 
$\mathbb{Q}T_x e_{\tilde{\chi}} \simeq M_r(\mathbb{Q}(\chi))$, where $r = dim(V)$.  When $V$ is thin, $V = \sum_i E_i^*V$ and each nonzero $E_i^*V$ has dimension $1$, so $dim(V)$ is equal to the dual diameter of $V$.  
\end{proof} 

\section{ Distinguishing Order $27$ schemes with their rational Terwilliger algebras } 

In this section we will illustrate how we have used Theorem \ref{ThinSchurIndex} to determine the algebraic structure of the rational Terwilliger algebras at every vertex of asymmetric rank $3$ association schemes of order $27$, and found that these do distinguish these schemes up to combinatorial isomorphism.  As there are $374$ of these schemes, and $27$ rational Terwilliger algebras for each one, we will not include all of the details.  Instead we will illustrate the necessary techniques by presenting the most interesting and difficult cases we encountered. 

Let $\mathbb{Q}T_x$ be the rational Terwilliger algebra of an asymmetric rank $3$ association scheme.  By Theorems \ref{ThinSchurIndex} and \ref{main}, we have that the non-primary simple components of $\mathbb{Q}T_x$ are either $\mathbb{Q}(\chi)$ or $M_2(\mathbb{Q}(\chi))$, and the central fields $\mathbb{Q}(\chi)$ are the fields $\mathbb{Q}(\theta)$ where $\theta$ is an eigenvalue of the subconstituent matrix $B_1$ defined in Equation (\ref{eq5}).   So the approach is to determine these central fields and compare how many of each occur. 

\begin{example} {\rm The largest collection of schemes of order $27$ whose complex Terwilliger algebras are all isomorphic has size $23$, for these the dimensions of their Terwilliger algebras, with frequency, are $(57^{24},49^3)$.  For these, every non-primary simple component will be a $2 \times 2$ matrix ring over its center.   

The first pair we will distinguish is {\tt as27no134} and {\tt as27no288}.  In both cases the simple components of the three rational Terwilliger algebras of dimension $49$ have central fields that are extensions of degree $2$ and $8$, and the quadratic extensions are the splitting field of $x^2+x+4$.  As it is harder to distinguish extensions of degree $8$ we look first at their collections of $24$ Terwilliger algebras of dimension $57$.  We consider the distribution of degrees of central field extensions in each case, always the sum of these degrees will be $12$, and identify the specific quadratic extensions.  

\smallskip
For {\tt as27no134}, these central fields and their freqencies are: 

$\bullet$ degree $12$ extension ($17$ times); 

$\bullet$ degree $10$ extension and the splitting field of $x^2+x+4$ ($1$ time) 

$\bullet$ degree $10$ extension and the splitting field of $x^2+x+3$ ($1$ time) 

$\bullet$ degree $8$ extension and degree $4$ extension ($4$ times) 

$\bullet$ degree $8$ extension and the splitting fields of $x^2+x+4$ and $x^2+x+1$ ($1$ time) 

\smallskip
For {\tt as27no288}, this central field distribution is: 

$\bullet$ degree $12$ extension ($17$ times); 

$\bullet$ degree $10$ extension and the splitting field of $x^2+x+6$ ($1$ times) 

$\bullet$ degree $10$ extension and the splitting field of $x^2+x+3$ ($1$ time) 

$\bullet$ degree $8$ extension and degree $4$ extension ($4$ times) 

$\bullet$ degree $8$ extension and splitting fields of $x^2+x+4$ and $x^2+x+1$ ($1$ time) 

\smallskip
We can conclude from this information that the rational Terwilliger algebras for these two association schemes are distinguished by the specific quadratic extensions occurring as centers of simple components.  The first one {\tt as27no134} has the splitting field $\mathbb{Q}(\sqrt{-11})$ of $x^2+x+3$ occurring as a center of one of the simple components of its rational Terwilliger algebras, while {\tt as27no288} does not, and instead has this center replaced by the splitting field $\mathbb{Q}(\sqrt{-23})$ of $x^2+x+6$.  

The next pair we will consider is {\tt as27no186} and {\tt as27no276}, which lies in the same collection as the previous pair.  For these the degrees of central fields for their $24$ Terwilliger algebras of dimension $57$ match: each has $17$ degree $12$ extensions and $7$ that have a degree $8$ and degree $4$ extension.  When we look at their three Terwilliger algebras of dimension $49$, {\tt as27no82} has a degree $8$ extension occurring with the splitting field $\mathbb{Q}(\sqrt{-15})$ of $x^2+x+4$ for two of them, and a degree $8$ extension occurring with the splitting field of $\mathbb{Q}(\sqrt{-11})$ of $x^2+x+3$ for the other.  For {\tt as27no276} it is the opposite: two have a degree $8$ extension and $\mathbb{Q}(\sqrt{-11})$ and the other has a degree $8$ extension occurring with $\mathbb{Q}(\sqrt{-15})$.  So the frequency of the quadratic extensions occurring as centers of simple components distinguishes the two lists of rational Terwilliger algebras. 
}\end{example} 

The last example gives the most sensitive technique required to distinguish asymmetric rank $3$ association schemes of order $27$ using their rational Terwilliger algebras. 

\begin{example} {\rm The three association schemes {\tt as27no11}, {\tt as27no106}, and {\tt as27no168} have isomorphic complex Terwilliger algebras: $25$ of dimension $57$, one of dimension $41$, and one of dimension $25$.  
{\tt as27no168} is distinguished from the others by the degrees of the central field extensions in its Terwilliger algebras, since it has $15$ degree $12$ extensions occurring as centers of its Terwilliger algebras of dimension $57$ and for the other two schemes this number is $20$.  For {\tt as27no11} and {\tt as27no106}, the degrees of all central field extensions occurring among simple components of their rational Terwilliger algebras match.  When we consider the central field extensions in their Terwilliger algebras of dimension $41$, the splitting field of $x^4+2x^3+6x^2+5x+2$ occurs in both.  The other central field occurring in both algebras is also of degree $4$.  For {\tt as27no11} it is the non-normal extension $L_1$ obtained by adjoining a root of $x^4+2x^3+9x^2+8x+11$ and for {\tt as27no106} it is the extension $L_2$ obtained by adjoining a root of $x^4+2x^3+9x^2+8x+13$.  We can use GAP to check that $L_1$ contains no root of $x^4+2x^3+9x^2+8x+13$, and therefore the fields $L_1$ and $L_2$ are not $\mathbb{Q}$-isomorphic.  So these two asssociation schemes are distinguished by just one difference in the degree $4$ extensions occurring in their rational Terwilliger algebra of dimension $41$. 
}\end{example}

In terms of rational representation theory, the semisimple algebras in the last example come quite close to being isomorphic.  It seems unlikely that rational Terwilliger algebras will be enough to distinguish the 98300 non-isomorphic asymmetric rank $3$ association schemes of order $31$, but this formidable task is yet to be attempted.  

\medskip
{\footnotesize
\noindent {\bf Data availability statement:} This manuscript has no associated data. 
}


\begin{thebibliography}{10}

\bibitem{ASpkg} J. Bamberg, A. Hanaki, and J. Lansdown, {\tt AssociationSchemes}, A GAP package for working with association schemes and homogeneous coherent configurations, Version 3.0.0 (2023). \\
{\tt http://www.jesselansdown.com/AssociationSchemes}  

\bibitem{BM95} E. Bannai and A. Munemasa, The Terwilliger algebras of group association schemes, {\it Kyushu J. Math.}, {\bf 49} (1995), 93--102.

\bibitem{Balmaceda-Oura94} J. M. P. Balmaceda and M. Oura, The Terwilliger algebras of the group association schemes of $S_5$ and $A_5$, {\it Kyushu J. Math.}, {\bf 48} (1994), 221--231. 

\bibitem{BST2010} G. Bhattacharyya, S.-Y. Song, and R. Tanaka, On Terwilliger algebras of wreath products of one-class association schemes, {\it J. Algebr. Comb.}, {\bf 31} (2010), no. 3, 455--466.

\bibitem{CGS78} P. J. Cameron, J. M. Goethals, and J. J. Seidel, Strongly regular graphs having strongly regular subconstituents, {\it J. Algebra}, {\bf 55} (1978), 257-280.

\bibitem{Caughman99} J. S. Caughman IV, The Terwilliger algebra for bipartite $P$- and $Q$-polynomial association schemes, {\it Discrete Math.}, {\bf 196} (1999), 65--95.

\bibitem{CRvol2} C. W. Curtis and I. Reiner, {\it Methods of Representation Theory, with Applications to Finite Groups and Orders, Vol. 2}, Wiley Interscience, John Wiley \& Sons, 1994. 

\bibitem{GAP4} The GAP Group, GAP -- Groups, Algorithms, and Programming, Version 4.12.2; 2022. \\
{\tt https://www.gap-system.org} 

\bibitem{HO2019} N. Hamid and M. Oura, Terwilliger algebras of some group association schemes, {\it Math. J. Okayama Univ.}, {\bf 61} (2019), 199–204.

\bibitem{Hanaki-ModularTerAlg} A. Hanaki, Modular Terwilliger algebras of association schemes, {\it Graphs Combin.}, {\bf 37} (2021), no. 5, 1521–1529.

\bibitem{HKMT-JCD2020} A. Hanaki, H. Kharaghani, A. Mohammadian, and B. Tayfeh-Rezaie, Classification of skew-Hadamard matrices of order 32 and association schemes of order 31, {\it J. Combin. Des.} {\bf 28} (2020), no. 6, 421-427. 

\bibitem{HKM2011} A. Hanaki, K. Kim, and Y. Maekawa, Terwilliger algebras of direct and wreath products of association schemes, {\it J. Algebra}, {\bf 343} (2011), 195--200. 

\bibitem{Hanaki-Miyamoto} A. Hanaki and I. Miyamoto, {\it Classification of Association Schemes}, (updated November 3, 2019). \\
{\tt http://math.shinshu-u.ac.jp/~hanaki/as/}

\bibitem{Maleki2024} R. Maleki, On the Terwilliger algebra of the group association scheme of  $C_n \rtimes C_2$, {\it Discrete Math.}, {\bf 347} (2), (2024), \#113773. 

\bibitem{MX2016} M. Muzychuk and B. Xu, Terwilliger algebras of wreath products of association schemes, {\it Linear Algebra Appl.}, {\bf 493} (2016), 146--163.

\bibitem{Tanabe97} K. Tanabe, The irreducible modules of the Terwilliger algebras of Doob schemes, {\it J. Algebr. Comb.}, {\bf 6} (1997), 173--195. 

\bibitem{T92} P. Terwilliger, The subconstituent algebra of an association scheme, I. {\it J. Algebraic Combin.} {\bf 1} (1992), 362-388.

\bibitem{T93b} P. Terwilliger, The subconstituent algebra of an association scheme, III. {\it J. Algebraic Combin.} {\bf 2} (1993), 177-210. 

\bibitem{TY94} M. Tomiyama and N. Yamazaki, The subconstituent algebra of a strongly regular graph, {\it Kyushu J. Math.}, {\bf 48} (1994), no. 2, 323-334. 

\end{thebibliography}
\end{document}